\documentclass[]{amsart}
\usepackage{graphicx}
\usepackage{latexsym}
\usepackage{hyperref}
\usepackage{amsfonts}
\usepackage[all]{xy}
\usepackage{amssymb, mathrsfs, amsfonts, amsmath}
\usepackage{amsbsy}
\usepackage{amsfonts}
\usepackage{setspace}
\makeatletter
\@namedef{subjclassname@2010}{%
  \textup{2010} Mathematics Subject Classification}
\makeatother

\usepackage{mathrsfs}
\usepackage{amsbsy}

\newtheorem{definition}{Definition}
\newtheorem{lemma}{Lemma}
\newtheorem{proposition}{Proposition}
\newtheorem{remark}{Remark}
\newtheorem{theorem}{Theorem}

\numberwithin{equation}{section}

\newcommand{\SP} [1]     {{\left\langle {{#1}} \right\rangle}}

\makeatletter
\@namedef{subjclassname@2020}{%
  \textup{2020} Mathematics Subject Classification}
\makeatother

\begin{document}
	
	\title[Critical metrics of the volume functional]{Integral and boundary estimates\\ for critical metrics of the volume functional}

	\author{Rafael Di\'ogenes}
	\author{Neilha Pinheiro}
	\author{Ernani Ribeiro Jr}

\address[R. Di\'ogenes]{UNILAB, Instituto de Ci\^encias Exatas e da Natureza, Rua Jos\'e Franco de Oliveira, 62790-970, Reden\c{c}\~ao / CE, Brazil.}\email{rafaeldiogenes@unilab.edu.br}

\address[N. Pinheiro]{Universidade Federal do Amazonas, Instituto de Ci\^encias Exatas - ICE, Departamento de Matem\'atica. Campus Coroado, 69080900 - Manaus / AM, Brazil.}\email{neilha@ufam.edu.br}

\address[E. Ribeiro Jr]{Universidade Federal do Cear\'a, Departamento  de Matem\'atica, Campus do Pici, Av. Humberto Monte, Bloco 914, 60455-760, Fortaleza / CE, Brazil}\email{ernani@mat.ufc.br}

	\thanks{R. Di\'ogenes was partially supported by CNPq/Brazil [Grant: 310680/2021-2]}
	\thanks{E. Ribeiro was partially supported by CNPq/Brazil [Grant: 309663/2021-0], CAPES/Brazil and FUNCAP/Brazil [Grant: PS1-0186-00258.01.00/21].}
	
	\begin{abstract}
		In this article, we investigate the geometry of critical metrics of the volume functional on compact manifolds with boundary. We use the generalized Reilly's formula to derive new sharp integral estimates for critical metrics of the volume functional on $n$-dimensional compact manifolds with boundary. As application, we establish new boundary estimates for such manifolds.	
		\end{abstract}
	
	\date{\today}
	
	
	
	\maketitle
	
\section{Introduction}
\label{SecInt}

A classical method to find canonical metrics on a given manifold is to investigate critical metrics which arise as solutions of the Euler-Lagrange equations for curvature functionals. For example, Einstein and Hilbert showed that the critical metrics of the scalar curvature functional on compact manifolds with unitary volume are Einstein (see \cite[Theorem 4.21]{Besse}). In a similar context and also motivated by a volume comparison result obtained by Fan, Shi and Tam \cite{FST} on asymptotically flat $3$-dimensional manifolds, Miao and Tam \cite{Miao-Tam 2011,Miao-Tam 2009} and Corvino, Eichmair and Miao \cite{Corvino-Eichmair-Miao} investigated the modified problem of finding stationary points for the volume functional constrained to the space of metrics of constant scalar curvature on a given compact manifold with boundary. 

In order to proceed, it is appropriate to fix some terminology (cf. \cite{Bach-Flat,Batista-Diogenes-Ranieri-Ribeiro JR}).  

\begin{definition} 
Let $(M^{n},\,g)$ be a connected compact Riemannian manifold with boundary $\partial M$ and dimension $n\ge 3.$ We say that $g$ is, for brevity, a {\it Miao-Tam critical metric} (or simply, {\it critical metric}), if there is a nonnegative smooth function $f$ on $M^n$ such that $f^{-1}(0)=\partial M$ satisfying the equation
\begin{equation}\label{eq-Miao-Tam}
\mathfrak{L}_{g}^{*}(f)=-(\Delta f)g+\nabla^2 f-fRic=g.
\end{equation} Here, $\mathfrak{L}_{g}^{*}$ is the formal $L^{2}$-adjoint of the linearization of the scalar curvature operator $\mathfrak{L}_{g}$. Moreover, $Ric,$ $\Delta$ and $\nabla^2$ stand for the Ricci tensor, the Laplacian operator and the Hessian on $(M^n,\,g),$ respectively.
\end{definition}

Miao and Tam \cite{Miao-Tam 2009} showed that such critical metrics defined in (\ref{eq-Miao-Tam}) arise as critical points of the volume functional on $M^n$ when restricted to the class of metrics $g$ with prescribed constant scalar curvature such that $g_{|_{T \partial M}}=\gamma$ for a prescribed Riemannian metric $\gamma$  on the boundary; see also \cite{Corvino-Eichmair-Miao}. Interestingly, it follows from \cite[Theorem 7]{Miao-Tam 2009} that connected Riemannian manifolds satisfying the critical metric equation (\ref{eq-Miao-Tam}) have necessarily constant scalar curvature $R.$ In particular, for any given metric $g$ in the space of the metrics with constant scalar curvature, the map $\mathfrak{L}_{g}^{*}$  defined from $C^{\infty}$ to $\mathcal{M}$ is an over determined-elliptic operator. Such metrics are relevant in understanding the influence of the scalar curvature in controlling the volume of a given manifold. In this scenario, Corvino, Eichmair and Miao \cite{Corvino-Eichmair-Miao} used the study of critical metrics to establish a deformation result which suggests that the information of the scalar curvature is not sufficient in giving volume comparison. In particular, it is known by the works \cite{Corvino-Eichmair-Miao,Miao-Tam 2009} that the Schoen's conjecture can not be generalized directly to manifolds with boundary if only the Dirichlet boundary condition is imposed (see also \cite{yuan}). Recall that the Schoen's conjecture \cite{Schoen2} asserts: {\it Let $(M^n,\,\overline{g})$ be a closed hyperbolic manifold and let $g$ be another metric on $M^n$ with scalar curvature $R(g) \geq R(\overline{g}),$ then $Vol(g) \geq Vol(\overline{g})$}. The $3$-dimensional case of the conjecture follows as a consequence of the works of Hamilton on nonsingular Ricci flow and Perelman on geometrization of $3$-manifolds.

Some explicit examples of critical me\-trics can be found in  \cite{Miao-Tam 2011,Miao-Tam 2009}. They include the spatial Schwarzschild metrics and AdS-Schwarzschild metrics restricted to certain domains containing their horizon and bounded by two spherically symmetric spheres. Besides, standard metrics on geodesic balls in space forms $\Bbb{R}^n,$ $\Bbb{H}^n$ or $\Bbb{S}^n$ are also critical metrics. These last ones are simply connected and have attracted a lot of attention in the last few years. As observed by Miao and Tam in \cite{Miao-Tam 2011}, it is interesting to know whether the standard metrics on geodesic balls in space forms are the only critical metrics on simply connected compact manifolds with connected boundary. There have been a lot of advances concerning the rigidity of critical metrics of the volume functional; see, e.g., \cite{Baltazar-Diogenes-Ribeiro,Baltazar-Ribeiro Jr. Ricci paralelo,Bach-Flat,BS,Batista-Diogenes-Ranieri-Ribeiro JR,Corvino-Eichmair-Miao,KS,Miao-Tam 2011,Miao-Tam 2009,SW}.

Boundary estimates are classical objects of study in geometry and physics. Besides being interesting on their own, such estimates play a fundamental role in proving classification results and discarding some possible new examples of special metrics on a given manifold. In recent years, it was established some useful boundary and volume estimates for critical metrics of the volume functional, as for example, isoperimetric and Shen-Boucher-Gibbons-Horowitz type inequalities (see, e.g., \cite{BBR,Baltazar-Diogenes-Ribeiro,BS,Batista-Diogenes-Ranieri-Ribeiro JR,Corvino-Eichmair-Miao,FY}). At the same time, it is natural to explore integral estimates in order to obtain new obstruction results. In this context, the Reilly's formula \cite{Reilly1977} has been shown to be a promising tool. Such a formula was used in solving interesting problems in differential geometry and has numerous applications. For instance, Ros \cite{Ros} used the Reilly’s formula to prove an integral inequality which was applied to show the Alexandrov's rigidity theorem for high order mean curvatures. Miao, Tam and Xie \cite{Miao-Tam-N.-Q} employed the Reilly's formula to obtain an integral inequality that relates the mean curvature $H$ and the second fundamental form $\mathbb{II}$ on the boundary $\partial \Omega$ of bounded domains $\Omega$ in the Euclidean space $\mathbb{R}^{n}.$ To be precise, by assuming that the mean curvature of the boundary $H>0,$ they proved that
\begin{equation}
\label{eqkl1} \int_{\partial \Omega} \left[\frac{(\Delta_{_{\partial \Omega}}\eta)^{2}}{H}-\mathbb{II}(\nabla_{_{\partial \Omega}}\eta, \nabla_{_{\partial \Omega}}\eta)\right]dS \geq 0,
\end{equation} for any smooth function $\eta$ on $\partial \Omega$, where $\nabla_{_{\partial \Omega}}$, $\Delta_{_{\partial \Omega}}$ and $dS$ denote the gradient, Laplacian and volume form on $\partial \Omega,$ respectively. Moreover, equality holds in (\ref{eqkl1}) if and only if $\eta = a_{0}+\sum_{i=1}^{n}a_{i}x_{i}$ for some constants $a_{0}, a_{1},...,a_{n}$. Here, $\{x_{1},x_{2},...,x_{n}\}$ are the standard coordinate functions on $\mathbb{R}^{n}$. In other words, (\ref{eqkl1}) can be seen as a stability inequality for Wang-Yau energy on $\partial \Omega,$ which is quite important in general relativity. For more details on this subject, see, e.g., \cite{Chen-Wang-Yau,Wang-Yau,Wang-Yau isometric}. Subsequently, Kwong and Miao \cite{Kwong-Miao} obtained a generalization of the estimate (\ref{eqkl1}) to the boundary of compact Riemannian manifolds whose metric is static. In \cite{Qiu-Xia}, Qiu and Xia established a generalized Reilly's formula that was used to give an alternative proof of the Alexandrov’s theorem and prove a new Heintze-Karcher inequality for Riemannian manifolds with boundary and sectional curvature bounded from below. We refer the reader to \cite{Kwong-Miao,Li-Xia,MW,X. Wang,Xia0} for further interesting applications.

In our first result, motivated by the above discussion, we will make use of the generalized Reilly's formula by Qiu and Xia \cite{Qiu-Xia} and a suitable boundary value problem in order to derive the following integral inequality for critical metrics of the volume functional, which is similar to Miao-Tam-Xie and Kwong-Miao estimates (cf. \cite[Corollary 3.1]{Miao-Tam-N.-Q} and \cite[Theorem 3]{Kwong-Miao}).

\begin{theorem}\label{th1} 
Let $\big(M^{n},\,g,\,f\big)$ be an $n$-dimensional connected compact oriented critical metric with connected boundary $\partial M.$ Suppose that $Ric \geq (n-1)\kappa g,$ where $\kappa$ is a constant. Then we have:
\begin{eqnarray}
\label{th1-eq02}
\frac{1}{H}\int_{\partial M}\Big[|\nabla_{_{\partial M}}\eta|^2 - (n-1)\kappa \eta^2 &+& H^2 \eta\langle \nabla u, \nabla f\rangle\Big] dS \nonumber\\&&\geq -\kappa \left(R-n(n-1)\kappa\right)\int_{M}f u^{2} dV_{g},
\end{eqnarray} for any function $\eta$ on $\partial M,$ where $H$ stands for the mean curvature of $\partial M$ and $u$ is a solution of
\begin{eqnarray}\label{The initial value problem}
\left\{\begin{array}{rc}
\Delta u +n\kappa u = 0 &\mbox{ in }\,\,\, M,\\
u= \eta & \mbox{ on }\,\,\, \partial M.
\end{array}\right.
\end{eqnarray} In particular, for $\kappa\leq 0,$ one has 
\begin{equation}
\label{eqgbn}
\int_{\partial M}\left(|\nabla_{_{\partial M}}\eta|^2 - (n-1)\kappa \eta^2 + H^2 \eta\langle \nabla u, \nabla f\rangle \right) dS\geq 0.
\end{equation} Moreover, equality holds in (\ref{eqgbn}) if and only if $(M^n,\,g)$ is isometric to a geodesic ball in the  simply connected space form $\Bbb{H}^n$ or $\Bbb{R}^n.$

\end{theorem}

\begin{remark}
We point out that, for $\kappa\leq 0,$ given any nontrivial $\eta$ on $\partial M,$ there exists a unique solution $u$ to the problem (\ref{The initial value problem}). Moreover, according to \cite[Theorem 4]{Reilly}, in the case $\kappa>0,$ one can still solve (\ref{The initial value problem}), provided that $(M^n,\,g)$ is not isometric to the standard hemisphere of radius $\kappa.$ In particular, given a critical metric $\big(M^{n},\,g,\,f\big),$ by assuming that $\eta=-\frac{n}{n-1},$ one easily verifies that $u=-\frac{Rf+n}{n-1}$ is a solution of (\ref{The initial value problem}) with $\kappa=\frac{R}{n(n-1)}.$
\end{remark}

In light of the above, it is natural to ask whether the assumption on the Ricci curvature $Ric\geq(n-1)\kappa g$ in Theorem \ref{th1} can be replaced by a lower bound condition on the scalar curvature, as for example, $R\geq n(n-1)\kappa.$ In such a situation, we have the following result.

\begin{theorem}\label{th2}
Let $\big(M^{n},\,g,\,f\big)$ be an $n$-dimensional connected compact oriented critical metric with connected boundary $\partial M.$ Suppose that $R\geq n(n-1)\kappa,$ where $\kappa$ is a non-positive constant. Then we have:
\begin{eqnarray}\label{th2-eq0}
H\int_{\partial M}\SP{\nabla f,\nabla u}^2 dS &-&\frac{(n-1)\kappa}{H}\int_{\partial M}\eta^2 dS \nonumber\\&\geq&\frac{1}{n-1}\int_M|\nabla u|^2 dV_g+n\kappa\int_M u^2\left(2n\kappa f-1\right)dV_g\nonumber\\
 && -(3n-2)\kappa\int_M f|\nabla u|^2dV_g,
\end{eqnarray} for any function $\eta$ on $\partial M$ and $u$ a solution of
		\begin{eqnarray}
		\label{eqjk78}
			\left\{\begin{array}{rc}
				\Delta u +n\kappa u = 0 &\mbox{in}\,\, M,\\
				u= \eta & \mbox{ on } \partial M.
			\end{array}\right.
		\end{eqnarray}
Moreover, equality holds in (\ref{th2-eq0}) if and only if
$R=n(n-1)\kappa$ and $\nabla^2u+\kappa ug=0.$
\end{theorem}

Before discussing our next result, we recall that the Yamabe constant for a Riemannian manifold $M^n$ with boundary $\partial M$ is given by

\begin{equation}
\label{Yamabeconst}
\mathcal{Y}(M,\partial M,[g]) = \inf_{0<\phi\in C^{\infty}(M)}\frac{\int_M\left(\frac{4(n-1)}{n-2}|\nabla\phi|^2+R\phi^2\right)dV_{g}+2\int_{\partial M}H\phi^2dS}{\left(\int_M|\phi|^\frac{2n}{n-2}dV_{g}\right)^\frac{n-2}{n}},
\end{equation} where $H$ is the mean curvature of $\partial M$ and $\phi$ is a positive smooth function on $M^n.$ We highlight that $\mathcal{Y}(M,\partial M,[g])$ is invariant under a conformal change of the metric $g.$ We refer to \cite{Brendle,EscobarJDG}  for a general discussion on this subject.

It is well-known that the Yamabe invariant plays a crucial role in the study of pres\-cribed me\-trics. In \cite{BS}, Barros and Silva showed that a critical metric of the volume functional $\big(M^{n},\,g,\,f)$ with connected Einstein boundary $\partial M$ of positive scalar curvature must satisfy $$|\partial M|^{\frac{2}{n-1}}\leq \frac{\mathcal{Y}(\Bbb{S}^{n-1},\,[g_{can}])}{C(R)},$$ where $C(R)=\frac{n-2}{n}R+\frac{n-2}{n-1}H^2$ is a positive constant and $\mathcal{Y}(\Bbb{S}^{n-1},\,[g_{can}])$ denotes the Yamabe constant of a standard sphere $\Bbb{S}^{n-1}.$ In the specific dimension $n=4,$ the Yamabe invariant alone is too weak to classify a given manifold. Accordingly, it is natural to impose an additional condition in order to obtain a classification theorem. A recent result due to Baltazar, Di\'ogenes and Ribeiro \cite{BDR}, inspired by the work in \cite{Catino 2016}, establishes a sharp integral curvature estimate involving the Yamabe constant for $4$-dimensional critical metrics of the volume functional with positive scalar curvature. In the same spirit, as a consequence of Theorem \ref{th2}, we shall derive the following boundary estimate involving the Yamabe constant for critical metrics of arbitrary dimension $n\ge 3,$ which can be also seen as an obstruction result for the existence of new examples of critical metrics.

\begin{theorem}
\label{corAthm}
Let $\big(M^{n},\,g,\,f\big)$ be an $n$-dimensional connected compact oriented critical metric with connected boundary $\partial M$ and non-negative scalar curvature. Then we have:

\begin{equation}
\label{plk456}
\mathcal{Y}(M,\partial M,[g])\Psi \leq \frac{4(n-1)^2R^2+4n(n^2-1)RH^2+2n^2(n-2)H^4}{(n-1)^2(n-2)H^3} |\partial M|,
\end{equation} where $\Psi=\left(\int_{M}(\Delta f)^\frac{2n}{n-2}dV_{g}\right)^{\frac{n-2}{n}}.$ Furthermore, if the equality holds in (\ref{plk456}), then $(M^n,\,g)$ has zero scalar curvature. 
\end{theorem}

In the work \cite[Proposition 2.5]{Corvino-Eichmair-Miao}, Corvino, Eichmair and Miao showed that the area of the boundary $\partial M$ of an $n$-dimensional scalar flat critical metric must have an upper bound depending on the volume of $M^n.$ In particular, it follows from (\ref{laplaciano}) that scalar flat critical metrics must satisfy $$\Delta f=-\frac{n}{n-1}.$$ Consequently, by the Stokes' formula, one sees that
\begin{eqnarray*}
Vol(M)=\frac{n-1}{nH}|\partial M|.
\end{eqnarray*} This raised the question about the cases of negative and positive scalar curvature. In \cite{Baltazar-Diogenes-Ribeiro}, Baltazar, Di\'ogenes and Ribeiro established a sharp isoperimetric inequality for critical metrics with non-negative scalar curvature. In spite of that, the isoperimetric constant obtained by them depends on the potential function $f.$ It would be interesting to see if such a constant can be improved to depend only on the dimension and mean curvature of the boundary. Other boundary estimates for critical metrics were obtained in, e.g., \cite{BBR,BS,Batista-Diogenes-Ranieri-Ribeiro JR,Corvino-Eichmair-Miao,FY}.

 In our next result, motivated by these aforementioned facts and the approach used in the proof of Theorem \ref{th1}, we have established the following estimate.

\begin{theorem}\label{th3}
Let $\big(M^{n},\,g,\,f\big)$ be an $n$-dimensional connected compact oriented critical metric with connected boundary $\partial M$ and positive scalar curvature. Then we have:
\begin{eqnarray}\label{th3-est01}
Vol(M)\geq\frac{n-1}{n}\sqrt{\frac{n(n+2)}{n(n+2)H^2+2(n-1)R}}\,|\partial M|.
\end{eqnarray}
Moreover, if the equality holds in (\ref{th3-est01}), then $(M^n,\,g)$ is isometric to a geodesic ball in a simply connected space form $\Bbb{S}^n.$
\end{theorem}

\begin{remark}
\label{remA}
We point out that in the case of negative scalar curvature, as an application of Theorem \ref{th2}, one
can further derive the estimate 

\begin{equation}
\label{eqV3}
 Vol(M)\geq \frac{(n-1)H}{nH^2 - R}\,|\partial M|.
\end{equation} For convenience, the proof of (\ref{eqV3}) will be presented in Proposition \ref{Propth3} of Section \ref{secproofs}. It turns out that this estimate is not sharp in the sense that it is not achieved by the geodesic ball in the hyperbolic space $\Bbb{H}^n.$
\end{remark}

\medskip

{\bf Acknowledgement.} We would like to thank the referee for his careful reading and valuable suggestions. Moreover, we would like to thank A. Barros and R. Batista for their interest in this work and helpful comments on an earlier version of the paper.

\section{Background}
\label{back}

In this section, we review some basic facts and present some key results that will play a crucial role in the proof of the main theorems.

We start by recalling that a Riemannian manifold $(M^n,\,g)$ is a critical metric of the volume functional (or simply, {\it critical metric}), if there exists a smooth function $f$ such that
\begin{equation}
\label{fund-equation}
-(\Delta f)g+\nabla^2 f-f Ric= g
\end{equation} for $f\geq0$ and $f^{-1}(0)=\partial M$ is the boundary of $M.$ In particular, tracing (\ref{fund-equation}) we deduce that the potential function $f$ also satisfies the equation
\begin{equation}\label{laplaciano}
\Delta f = -\frac{Rf+n}{n-1}.
\end{equation} From this, we have

\begin{equation}
\label{eqVstaic2}\nabla^2  f-fRic=-\frac{Rf+1}{n-1}g
\end{equation}
and
\begin{equation}
\label{p1a}
f\mathring{Ric}=\mathring{\nabla}^2 f,
\end{equation} where $\mathring{T}=T-\frac{{\rm tr}T}{n}g$ stands for the traceless of tensor $T.$

It should be also mentioned that, choosing appropriate coordinates, $f$ and $g$ are analytic. Thus, the set of regular points of $f$ is dense in $M^n$ (see \cite[Proposition 2.1]{Corvino-Eichmair-Miao}). Besides, at regular points of $f,$ the vector field $\nu=-\frac{\nabla f}{|\nabla f|}$ is normal to $\partial M$ and it is known from \cite[Theorem 7]{Miao-Tam 2009} that $|\nabla f|$ is constant (non null) on each connected component of $\partial M.$ Furthermore, the second fundamental form of $\partial M$ is given by

$$\mathbb{II}(e_{i},e_{j})=\langle \nabla_{e_{i}}\nu, e_{j}\rangle,$$ where $\{e_{1},\ldots, e_{n-1}\}$ is an orthonormal frame on $\partial M.$ Thus, one obtains from (\ref{eqVstaic2}) that 

\begin{equation*}
\mathbb{II}(e_{i},e_{j})=-\left\langle \nabla_{e_{i}} \frac{\nabla f}{|\nabla f|}, \, e_{j}\right\rangle=\frac{1}{(n-1)|\nabla f|}g_{ij}.
\end{equation*} Consequently, the mean curvature is constant $H=\frac{1}{|\nabla f|}$ and therefore, $\partial M$ is totally umbilical. For more details, see \cite[Sec. 3]{Batista-Diogenes-Ranieri-Ribeiro JR}.

The following generalized Reilly's formula, obtained previously by Qiu and Xia \cite{Qiu-Xia}, will be very useful.

\begin{proposition}[\cite{Qiu-Xia}]\label{prop-qui-xia}
Let $(M^n,\,g)$ be an $n$-dimensional, compact Riemannian manifold with boundary $\partial M.$ Given two functions $f$ and $u$ on $M^n$ and a constant $\kappa$, we have 

\begin{eqnarray*}
&&\int_{M} f \left((\Delta u + \kappa nu)^{2}-|\nabla^{2}u+\kappa ug|^{2}\right)dV_g=(n-1)\kappa \int_{M}(\Delta f +n\kappa f)u^{2}\,dV_g \nonumber \\ &&+\int_{M}\left(\nabla ^{2}f-(\Delta f)g-2(n-1)\kappa f g+f Ric\right)(\nabla u, \nabla u)\,dV_g \nonumber \\
&&+\int_{\partial M}f \left[2\left(\frac{\partial u}{\partial \nu}\right)\Delta_{_{\partial M}}u+H\left(\frac{\partial u}{\partial \nu}\right)^{2}+ \mathbb{II}(\nabla_{_{\partial M}}u, \nabla_{_{\partial M}} u)+2(n-1)\kappa \left(\frac{\partial u}{\partial \nu}\right)u\right]dS\nonumber\\
&&+ \int_{\partial M}\frac{\partial f}{\partial \nu}\left(|\nabla _{_{\partial M}}u|^{2}-(n-1)\kappa u^{2}\right) dS,
\end{eqnarray*} where $H$ and  $\mathbb{II}$ stand for the mean curvature and second fundamental form of $\partial M,$ respectively.
 \end{proposition}

The classical Reilly’s formula is obtained by considering $f=1$ and $\kappa=0$ in the above expression. For the reader's convenience, we shall provide a proof of Proposition \ref{prop-qui-xia} here (cf. \cite[Proposition 1]{Kwong-Miao}).

\begin{proof}
 To begin with, upon integrating by parts over $M,$ one sees that

	\begin{eqnarray*}
	\int_M\SP{\nabla f,\nabla|\nabla u|^2}dV_{g}&=& \frac{3}{2}\int_{M}\langle \nabla f,\,\nabla |\nabla u|^2 \rangle dV_{g}-\frac{1}{2}\int_{M}\langle \nabla f,\nabla |\nabla u|^2 \rangle dV_{g} \nonumber\\&=&\frac{3}{2}\int_{M}\langle \nabla f,\,\nabla |\nabla u|^2 \rangle dV_{g}-\int_{M}\nabla^2u(\nabla u, \nabla f) dV_{g}\nonumber\\ &=&-\frac{3}{2}\int_M (\Delta f)|\nabla u|^2 dV_{g}+\frac{3}{2}\int_{\partial M}\frac{\partial f}{\partial\nu}|\nabla u|^2 dS\nonumber
	\\&&-\int_{M}\nabla^2u(\nabla u, \nabla f) dV_{g},
		\end{eqnarray*} and taking into account that 
\begin{equation}
\label{kljnm0}
\nabla^2f(\nabla u,\nabla u)={\rm div}\,(\langle\nabla f,\nabla u\rangle\nabla u)-\Delta u\langle\nabla f,\nabla u\rangle-\nabla^2u(\nabla u,\nabla f),
\end{equation} one obtains that
		
\begin{eqnarray}
	\label{tgh9}
	\int_M\SP{\nabla f,\nabla|\nabla u|^2}dV_{g} &=&-\frac{3}{2}\int_M (\Delta f)|\nabla u|^2 dV_{g}+\frac{3}{2}\int_{\partial M}\frac{\partial f}{\partial\nu}|\nabla u|^2 dS\nonumber
	\\
	&&-\int_{\partial M}\SP{\nabla f,\nabla u}\frac{\partial u}{\partial\nu}dS+\int_M\nabla^2f(\nabla u,\nabla u)dV_{g}\nonumber\\
	&&+\int_M\SP{\nabla f,\nabla u}\Delta u dV_{g}.
	\end{eqnarray} Now, observe that

	\begin{equation*}
	   \frac{1}{2}\Delta(f|\nabla u|^2)=
	   \frac{1}{2}(\Delta f)|\nabla u|^2+\frac{1}{2}f\Delta|\nabla u|^2+\SP{\nabla f, \nabla |\nabla u|^2}.
	\end{equation*} This jointly with (\ref{tgh9}) and the Bochner's formula:
	
	$$\frac{1}{2}\Delta |\nabla u|^2 = Ric(\nabla u,\,\nabla u)+|\nabla^2 u|^2 +\langle \nabla u,\,\nabla \Delta u\rangle,$$ gives
	\begin{eqnarray}\label{th1-eq20}
&&\frac{1}{2}\int_{\partial M}\frac{\partial}{\partial\nu}(f|\nabla u|^2)dS -\int_Mf\left(|\nabla^2u|^2+Ric(\nabla u, \nabla u)+\SP{\nabla\Delta u,\nabla u}\right)dV_{g}\nonumber\\
&=&-\int_M(\Delta f)|\nabla u|^2 dV_{g}
+\frac{3}{2}\int_{\partial M}\frac{\partial f}{\partial\nu}|\nabla u|^2dS -\int_{\partial M}\SP{\nabla f,\nabla u}\frac{\partial u}{\partial\nu}dS \nonumber\\&&+\int_M\nabla^2f(\nabla u,\nabla u)dV_{g}+\int_M\SP{\nabla f,\nabla u}\Delta u\, dV_{g}.	
	\end{eqnarray} 
	
	In another direction, upon integrating by parts, we achieve
	\begin{eqnarray*}
\int_Mf\SP{\nabla\Delta u,\nabla u}dV_{g}&=&-\int_Mf(\Delta u)^2 dV_{g}-
\int_M\SP{\nabla f,\nabla u}\Delta u dV_{g}+\int_{\partial M}f(\Delta u)\frac{\partial u}{\partial\nu}dS,
	\end{eqnarray*} moreover, by using Fermi coordinates, with $\nabla_\nu \nu=0,$ it is not difficult to check that  

\begin{equation*}
\frac{1}{2}\frac{\partial}{\partial\nu}|\nabla u|^2=
\SP{\nabla_{_{\partial M}}u,\nabla_{_{\partial M}}\Big(\frac{\partial u}{\partial\nu}\Big)}-\mathbb{II}(\nabla_{_{\partial M}}u,\nabla_{_{\partial M}}u)
+\frac{\partial u}{\partial\nu}\Big(\Delta u-\Delta_{_{\partial M}}u-H\frac{\partial u}{\partial\nu}\Big).
	\end{equation*} Plugging these two above expressions into (\ref{th1-eq20}) yields

\begin{eqnarray*}\label{th1-eq3}
&&\int_{\partial M}f\Big[-\mathbb{II}(\nabla_{_{\partial M}}u,\nabla_{_{\partial M}}u)
+\frac{\partial u}{\partial\nu}\Big(-2\Delta_{_{\partial M}}u-H\frac{\partial u}{\partial\nu}\Big)\Big]dS
-\int_{\partial M}\frac{\partial f}{\partial\nu}|\nabla_{_{\partial M}}u|^2 dS\nonumber\\
&&=\int_M f\left(|\nabla^2u|^2-(\Delta u)^2 \right)dV_{g}
+\int_M\left(-(\Delta f)g+\nabla^2f+fRic\right)(\nabla u,\nabla u)dV_{g}.
\end{eqnarray*} To conclude, it suffices to use the fact

\begin{eqnarray*}
&&\int_{M} f\left(|\nabla^2 u|^2-\left(\Delta u\right)^2 \right) dV_{g} = \int_{M}f\left[|\nabla^2 u+\kappa u g|^2 -\left(\Delta u+n\kappa u\right)^2 \right] dV_{g}\nonumber\\&& +(n-1)\kappa \left[\int_{\partial M}\left(2fu\frac{\partial u}{\partial\nu}-u^2 \frac{\partial f}{\partial \nu}\right)dS + \int_{M}\left((\Delta f)u^2 -2 f|\nabla u|^2 \right)dV_{g}\right]\nonumber\\&&+n(n-1)\kappa^2 \int_{M} f u^{2} dV_{g}.
\end{eqnarray*} So, the proof is finished. 
\end{proof}

Next, we shall establish a key lemma that will be used in the proofs of Theorems \ref{th1}, \ref{th2} and \ref{th3}.

\begin{lemma}
\label{lemplk1}
Let $\big(M^{n},\,g,\,f\big)$ be an $n$-dimensional connected compact oriented critical metric with connected boundary $\partial M.$ Then we have:
\begin{eqnarray}
\label{plk1}
\frac{1}{H}\int_{\partial M}\left(|\nabla_{_{\partial M}}\eta|^2 -(n-1)\kappa \eta^2 \right) dS &=& \int_{M}f|\nabla^2 u +\kappa u g|^2 dV_{g}+\int_{M}|\nabla u|^2 dV_{g}\nonumber\\&&+2\int_{M}f\Big[Ric-(n-1)\kappa g\Big](\nabla u, \nabla u) dV_{g}\nonumber\\&& -\kappa(R-n(n-1)\kappa)\int_{M}f u^2 dV_{g}\nonumber\\&& -n\kappa \int_{M}u^2 dV_{g},
\end{eqnarray}  where $\eta$ is any function on $\partial M,$ $H$ stands for the mean curvature of $\partial M$ and $u$ is a solution of
\begin{eqnarray}\label{The initial value problemAS}
\left\{\begin{array}{rc}
\Delta u +n\kappa u = 0 &\mbox{ in }\,\,\, M,\\
u= \eta & \mbox{ on }\,\,\, \partial M.
\end{array}\right.
\end{eqnarray}

\end{lemma}
\begin{proof}
Initially, we use Proposition \ref{prop-qui-xia} and the fact that $f=0$ on $\partial M$ to infer

\begin{eqnarray}\label{th1-eq1}
\int _{M} f \left[(\Delta u + \kappa nu)^{2}-|\nabla^{2}u+\kappa ug|^{2}\right]dV_g&=&(n-1)\kappa \int_{M}(\Delta f +n\kappa f)u^{2}dV_g\nonumber\\
&&+\int_{M}\left[\nabla ^{2}f-(\Delta f)g-fRic\right](\nabla u, \nabla u)dV_g \nonumber\\&&+2\int_{M}f\left[Ric-(n-1)\kappa g\right](\nabla u, \nabla u)dV_g\nonumber\\
&&+\int_{\partial M}\frac{\partial f}{\partial \nu}\left[|\nabla _{_{\partial M}}u|^{2}-(n-1)\kappa u^{2}\right] dS. 
\end{eqnarray} Substituting (\ref{fund-equation}) and (\ref{laplaciano}) into (\ref{th1-eq1}) yields
\begin{eqnarray}\label{th1-eq2}
\int_{M} f \left[(\Delta u + \kappa nu)^{2}\right.&-&\left.|\nabla^{2}u+\kappa ug|^{2}\right]dV_g=2\int_{M}f\left[Ric-(n-1)\kappa g\right](\nabla u, \nabla u)dV_g\nonumber\\
&+&\int_{M}|\nabla u|^2 dV_g-\kappa\left[R-n(n-1)\kappa\right]\int_{M}fu^{2}dV_{g}-n\kappa \int_{M}u^{2}dV_g\nonumber\\
&+&\int_{\partial M}\frac{\partial f}{\partial \nu}\left[|\nabla _{_{\partial M}}u|^{2}-(n-1)\kappa u^{2}\right]dS.
\end{eqnarray} Taking into account that $\nu = -\frac{\nabla f}{|\nabla f|}$ is the outward unit normal to $\partial M$ and $H=\frac{1}{|\nabla f|}$ is the mean curvature of $\partial M$ with respect to $\nu,$ one sees that $\frac{\partial f}{\partial \nu}= -|\nabla f|= -\frac{1}{H}.$ Hence, applying (\ref{The initial value problemAS}) and  (\ref{th1-eq2}), one obtains (\ref{plk1}), which gives the desired result. 

\end{proof}

\vspace{0.50cm}
Now we are ready to present the proofs of the main results.
\vspace{0.50cm}

\section{Proof of the Main Results}
\label{secproofs}

In this section, we will present the proofs of Theorems \ref{th1}, \ref{th2}, \ref{corAthm} and \ref{th3}.

\subsection{Proof of Theorem \ref{th1}}
\begin{proof} First of all, a direct computation shows
\begin{equation*}
{\rm div} (u\nabla u)=u\Delta u+|\nabla u|^2 =-n\kappa u^2 +|\nabla u|^2,
\end{equation*} where we have used (\ref{The initial value problem}). Hence, on integrating over $M^n,$ one has 

\begin{equation*}
\int_{M}\left(|\nabla u|^2 -n\kappa u^2\right) dV_{g} = \int_{M} {\rm div}(u\nabla u) dV_g=\int_{\partial M}u\langle \nabla u, \nu\rangle dS =-H\int_{\partial M}\eta \langle \nabla u,\nabla f\rangle dS.
\end{equation*} This, together with Lemma \ref{lemplk1}, gives

\begin{eqnarray}
\label{plk123}
\frac{1}{H}\int_{\partial M}\Big[|\nabla_{_{\partial M}}\eta|^2 &-& (n-1)\kappa \eta^2 + H^2 \eta\langle \nabla u, \nabla f\rangle\Big] dS\nonumber\\&=& \int_{M}f|\nabla^2 u+\kappa u g|^2 dV_{g}+2\int_{M}f\Big[Ric-(n-1)\kappa g\Big](\nabla u,\nabla u) dV_{g}\nonumber\\&& -\kappa (R-n(n-1)\kappa)\int_{M}f u^2 dV_{g}.
\end{eqnarray} Thereby, since $Ric \geq (n-1)\kappa g,$ one obtains that

$$\frac{1}{H}\int_{\partial M}\Big[|\nabla_{_{\partial M}}\eta|^2 - (n-1)\kappa \eta^2 + H^2 \eta\langle \nabla u, \nabla f\rangle\Big] dS\geq  -\kappa (R-n(n-1)\kappa)\int_{M}f u^2 dV_{g}.$$ Hence, (\ref{th1-eq02}) is proved.

We now deal with the equality case. Observe that our assumption also implies that $R\geq n(n-1)\kappa$ and then, for $\kappa \leq0,$ one concludes from (\ref{plk123}) that
 
 \begin{equation}\label{th1-eq6}
		\int_{\partial M}\left(|\nabla_{_{\partial M}}\eta|^2 - (n-1)\kappa \eta^2 + H^2 \eta\langle \nabla u, \nabla f\rangle \right) dS\geq 0.
		\end{equation} Whence, equality holds in (\ref{th1-eq6}) if and only if $$Ric(\nabla u,\nabla u)=(n-1)\kappa |\nabla u|^2,$$
		 \begin{equation}
		 \label{lki90}
		 R=n(n-1)\kappa
		 \end{equation} and $$\nabla^{2}u+u\kappa g=0.$$ Besides, since $Ric\geq (n-1)\kappa g,$ one obtains from (\ref{lki90}) that $$Ric=(n-1)\kappa g=\frac{R}{n}g,$$ that is, $g$ is an Einstein metric. To finish the proof, we apply Theorem 1.1 of \cite{Miao-Tam 2011} to conclude that $(M^n,\,g)$ is isometric to a geodesic ball in a simply connected space form $\Bbb{H}^n$ or $\Bbb{R}^n.$
	\end{proof}

\subsection{Proof of Theorem \ref{th2}}

\begin{proof}
To begin with, one inserts the critical equation (\ref{fund-equation}) into the identity from Lemma \ref{lemplk1} in order to obtain

\begin{eqnarray}\label{th2-eq1}
\frac{1}{H}\int_{\partial M}\left(|\nabla_{\partial M}\eta|^{2}-(n-1)\kappa \eta^{2}\right) dS &=&\int _{M}f|\nabla^{2}u+\kappa ug|^{2}dV_g-\int_{M}|\nabla u|^{2}dV_g\nonumber\\
 &&+2\int_{M}\left[(-\Delta f)g+\nabla^2f\right](\nabla u, \nabla u)dV_g\nonumber\\
 &&-2(n-1)\kappa \int_M f|\nabla u|^2 dV_g\nonumber\\ 
 &&-\kappa\left(R-n(n-1)\kappa\right)\int_{M}fu^{2}dV_g-n\kappa\int_{M}u^{2} dV_g.
\end{eqnarray} 

On the other hand, with aid of (\ref{kljnm0}), one sees that 

\begin{eqnarray*}
\int_M\nabla^2f(\nabla u,\nabla u) dV_g&=&\int_{\partial M}\langle\nabla f,\nabla u\rangle\langle\nabla u,\nu\rangle dS-\int_M \Delta u\langle\nabla f,\nabla u\rangle dV_g\\
 &&-\frac{1}{2}\int_M\langle\nabla|\nabla u|^2,\nabla f\rangle dV_g\\
 &=&\int_{\partial M}\langle\nabla f,\nabla u\rangle\langle\nabla u,\nu\rangle dS-\int_M\Delta u\langle\nabla f,\nabla u\rangle dV_g\\
 &&-\frac{1}{2}\int_{\partial M}|\nabla u|^2\langle\nabla f,\nu\rangle dS+\frac{1}{2}\int_M|\nabla u|^2\Delta f dV_g.
\end{eqnarray*} Rearranging terms we get

\begin{eqnarray*}
\int_M\left[(-\Delta f)g+\nabla^2f\right](\nabla u,\nabla u)dV_g&=&\int_{\partial M}\langle\nabla f,\nabla u\rangle\langle\nabla u,\nu\rangle dS-\int_M\langle\nabla f,\nabla u\rangle\Delta u\, dV_g\nonumber \\
&&-\frac{1}{2}\int_{\partial M}|\nabla u|^2\langle\nabla f,\nu\rangle dS-\frac{1}{2}\int_M|\nabla u|^2\Delta f\,dV_g.
\end{eqnarray*} This substituted into (\ref{th2-eq1}) gives

\begin{eqnarray}\label{th2-eq3}
\frac{1}{H}\int_{\partial M}\left(|\nabla_{\partial M}\eta|^2-(n-1)\kappa\eta^2\right)dS &=&\int_{M}f|\nabla^2u+\kappa ug|^2 dV_g-\int_M |\nabla u|^2 dV_g\nonumber\\
 &&+2\int_{\partial M}\SP{\nabla f,\nabla u}\SP{\nabla u,\nu}dS-2\int_{M}\SP{\nabla f,\nabla u}\Delta u dV_g\nonumber\\
 &&-\int_{\partial M}|\nabla u|^2\SP{\nabla f, \nu}dS
-\int_M |\nabla u|^2\Delta f dV_g\\
 &&-2(n-1)\kappa\int_M f|\nabla u|^2 dV_g\nonumber\\
 &&-\kappa\left(R-n(n-1)\kappa\right)\int_Mfu^2 dV_g-n\kappa\int_Mu^2 dV_g.\nonumber
\end{eqnarray} By using that $\nu=-\frac{\nabla f}{|\nabla f|},$ $\Delta u=-n\kappa u,$  $H=\frac{1}{|\nabla f|}$ and (\ref{laplaciano}), one obtains that

\begin{eqnarray}\label{th2-eq4}
\frac{1}{H}\int_{\partial M}[|\nabla_{\partial M}\eta|^2-(n-1)\kappa \eta^2]dS
 &=&\int_{M}f|\nabla^2u+\kappa ug|^2dV_g+\frac{1}{n-1}\int_M|\nabla u|^2 dV_g\nonumber\\
 &&-2H\int_{\partial M}\SP{\nabla f,\nabla u}^2 dS+2n\kappa \int_{M}u\SP{\nabla f,\nabla u}dV_g\nonumber\\
 &&+\frac{1}{H}\int_{\partial M}|\nabla u|^2 dS
+\frac{R-2(n-1)^2\kappa}{n-1}\int_Mf|\nabla u|^2 dV_g\\
&&-\kappa\left(R-n(n-1)\kappa\right)\int_M fu^2dV_g-n\kappa\int_M u^2 dV_g.\nonumber
\end{eqnarray}

Proceeding, it is easy to check from (\ref{eqjk78}) that
\begin{eqnarray*}
{\rm div}\,(fu\nabla u)&=&-n\kappa fu^2+f|\nabla u|^2+u\SP{\nabla f,\nabla u}.
\end{eqnarray*} This leads to 
\begin{eqnarray}\label{th2-eq5}
\int_M u\SP{\nabla f,\nabla u}dV_g&=&n\kappa \int_Mfu^2dV_g-\int_Mf|\nabla u|^2dV_g,
\end{eqnarray} where we have used that $f=0$ on $\partial M.$ Furthermore, on $\partial M,$ we have

\begin{eqnarray*}
|\nabla u|^2=\SP{\nabla u,\nu}^2+|\nabla_{\partial M}u|^2=H^2\SP{\nabla f,\nabla u}^2+|\nabla_{\partial M}\eta|^2,
\end{eqnarray*} so that
\begin{eqnarray}\label{th2-eq6}
\int_{\partial M}|\nabla_{\partial M}\eta|^2 dS-\int_{\partial M}|\nabla u|^2 dS=-H^2\int_{\partial M}\SP{\nabla f,\nabla u}^2 dS.
\end{eqnarray} Hence, substituting (\ref{th2-eq5}) and (\ref{th2-eq6}) into (\ref{th2-eq4}) yields
\begin{eqnarray}
\label{rf689}
H\int_{\partial M}\SP{\nabla f,\nabla u}^2 dS-\frac{(n-1)\kappa}{H}\int_{\partial M}\eta^2 dS&=&\int_{M}f|\nabla^2u+\kappa ug|^2 dV_g+\frac{1}{n-1}\int_M|\nabla u|^2 dV_g\nonumber\\
 &&+2n^2\kappa^2\int_Mfu^2dV_g-n\kappa\int_Mu^2 dV_g\nonumber\\
 &&+\frac{R-2(n-1)(2n-1)\kappa}{n-1}\int_M f|\nabla u|^2dV_g\nonumber\\
&&-\kappa\left(R-n(n-1)\kappa\right)\int_M fu^2 dV_g.
\end{eqnarray} Since $R\geq n(n-1)\kappa$ and $f>0$ in $M,$ it holds that

\begin{eqnarray*}
H\int_{\partial M}\SP{\nabla f,\nabla u}^2 dS-\frac{(n-1)\kappa}{H}\int_{\partial M}\eta^2 dS&\geq&\frac{1}{n-1}\int_M |\nabla u|^2 dV_g-n\kappa\int_M u^2 dV_g\nonumber\\
 &&+2n^2\kappa^2 \int_Mfu^2 dV_g\nonumber\\&&-(3n-2)\kappa\int_M f|\nabla u|^2 dV_g.\nonumber
\end{eqnarray*} In particular, observe that the equality holds in the above inequality if and only if $R=n(n-1)\kappa$ and $\nabla^2u+\kappa u g=0.$ So, the proof is finished. 
\end{proof}

\subsection{Proof of Theorem \ref{corAthm}}

\begin{proof}
First of all, one easily verifies that under the hypotheses of Theorem \ref{corAthm}, the derivation of (\ref{rf689}) holds. Moreover, observe that, for 
\begin{equation}
\label{eqkl120}
\kappa=\frac{R}{n(n-1)}\,\,\,\,\hbox{and}\,\,\,\,\,u=\Delta f=-\frac{Rf+n}{n-1},
\end{equation} we have 
	\begin{eqnarray}
			\left\{\begin{array}{rc}
				\Delta u +n\kappa u = 0 &\mbox{in}\,\, M,\\
				u= -\frac{n}{n-1}& \mbox{ on } \partial M,
			\end{array}\right.
		\end{eqnarray} that is, $u$ satisfies the problem (\ref{eqjk78}). This fact jointly with (\ref{rf689}) allows us to infer

\begin{eqnarray}\label{eq1112}
\frac{R^2-nH^2R}{(n-1)^2H^3}|\partial M|&=&\int_Mf|\mathring{\nabla^2}\Delta f|^2dV_g+\frac{R^2}{(n-1)^3}\int_M|\nabla f|^2dV_g+\frac{2R^2}{(n-1)^2}\int_Mf(\Delta f)^2dV_g\nonumber\\
 &&-\frac{R}{n-1}\int_M(\Delta f)^2dV_g-\frac{(3n-2)R^3}{n(n-1)^3}\int_Mf|\nabla f|^2dV_g\nonumber\\
 &=&\int_Mf|\mathring{\nabla^2}\Delta f|^2dV_g+\frac{R^2}{(n-1)^3}\int_M|\nabla f|^2dV_g+\frac{2R^2}{(n-1)^2}\int_Mf(\Delta f)^2dV_g\\
 &&-\frac{R}{n-1}\int_M(\Delta f)^2dV_g+\frac{(3n-2)R^3}{2n(n-1)^3}\int_Mf^2\Delta fdV_g,\nonumber
\end{eqnarray} where we have used integration by parts.

From the assumption $R\geq0$ and choosing $\phi=-\Delta f>0$ in (\ref{Yamabeconst}), we then deduce the following inequality
\begin{eqnarray}\label{eq1113}
\frac{R^2}{(n-1)^2}\int_M|\nabla f|^2dV_g&\geq&\frac{n-2}{4(n-1)}\mathcal{Y}(M,\partial M,[g])\Psi-\frac{(n-2)R}{4(n-1)}\int_M(\Delta f)^2dV_g\nonumber\\
 &&-\frac{n^2(n-2)H}{2(n-1)^3}|\partial M|,
\end{eqnarray} where $\Psi=\left(\int_M(\Delta f)^\frac{2n}{n-2}dV_g\right)^\frac{n-2}{n}.$

Plugging (\ref{eq1113}) into (\ref{eq1112}) yields

\begin{eqnarray*}
\frac{n-2}{4(n-1)^2}\mathcal{Y}(M,\partial M,[g])\Psi&\leq&\frac{R^2-nH^2R}{(n-1)^2H^3}|\partial M|-\int_Mf|\mathring{\nabla^2}\Delta f|^2dV_g\\
 &&+\frac{(n-2)R}{4(n-1)^2}\int_M(\Delta f)^2dV_g+\frac{n^2(n-2)H}{2(n-1)^4}|\partial M|\\
 &&-\frac{2R^2}{(n-1)^2}\int_Mf(\Delta f)^2dV_g+\frac{R}{n-1}\int_M(\Delta f)^2dV_g\\
 &&-\frac{(3n-2)R^3}{2n(n-1)^3}\int_Mf^2\Delta fdV_g\\
 &=&\frac{2(n-1)^2R(R-nH^2)+n^2(n-2)H^4}{2(n-1)^4H^3}|\partial M|\\
 &&-\int_Mf|\mathring{\nabla^2}\Delta f|^2dV_g+\frac{(5n-6)R}{4(n-1)^2}\int_M(\Delta f)^2dV_g\\
 &&-\frac{2R^2}{(n-1)^2}\int_Mf(\Delta f)^2dV_g-\frac{(3n-2)R^3}{2n(n-1)^3}\int_Mf^2\Delta fdV_g.
\end{eqnarray*} By using (\ref{laplaciano}) and  rearranging terms, one sees that

\begin{eqnarray}
\label{plkj1f}
\frac{n-2}{4}\mathcal{Y}(M,\partial M,[g])\Psi&\leq&\frac{2(n-1)^2R(R-nH^2)+n^2(n-2)H^4}{2(n-1)^2H^3}|\partial M|\nonumber\\
 &&-(n-1)^2\int_Mf|\mathring{\nabla^2}\Delta f|^2dV_g+\frac{(5n-6)R}{4}\int_M(\Delta f)^2dV_g\nonumber\\
 &&+\frac{(n+2)R^3}{2n(n-1)}\int_Mf^2\Delta fdV_g+\frac{2nR^2}{n-1}\int_Mf\Delta f dV_g.
\end{eqnarray} A direct computation using (\ref{laplaciano}) guarantees that $$R f\Delta f=-(n-1)(\Delta f)^2-n\Delta f,$$ which substituted into (\ref{plkj1f}) gives

\begin{eqnarray*}
\frac{n-2}{4}\mathcal{Y}(M,\partial M,[g])\Psi&\leq&\frac{2(n-1)^2R(R-nH^2)+n^2(n-2)H^4}{2(n-1)^2H^3}|\partial M|\\
 &&-(n-1)^2\int_Mf|\mathring{\nabla^2}\Delta f|^2dV_g+\frac{(5n-6)R}{4}\int_M(\Delta f)^2dV_g\\
 &&+\frac{(n+2)R^3}{2n(n-1)}\int_Mf^2\Delta fdV_g-2nR\int_M(\Delta f)^2dV_g\nonumber\\&&-\frac{2n^2R}{n-1}\int_M\Delta f dV_g.
\end{eqnarray*} Hence, we arrive at

\begin{eqnarray*}
\frac{n-2}{4}\mathcal{Y}(M,\partial M,[g])\Psi&\leq&\frac{2(n-1)^2R^2+2n(n^2-1)RH^2+n^2(n-2)H^4}{2(n-1)^2H^3}|\partial M|\nonumber\\
 &&-(n-1)^2\int_Mf|\mathring{\nabla^2}\Delta f|^2dV_g-\frac{3(n+2)R}{4}\int_M(\Delta f)^2dV_g\nonumber\\&&+\frac{(n+2)R^3}{2n(n-1)}\int_Mf^2\Delta fdV_g.\nonumber
\end{eqnarray*} Again, since $R\geq0,$ we have from (\ref{laplaciano}) that $\Delta f\leq0$ and therefore, we obtain
\begin{equation}\label{eq1115}
\mathcal{Y}(M,\partial M,[g])\Psi\leq\frac{4(n-1)^2R^2+4n(n^2-1)RH^2+2n^2(n-2)H^4}{(n-1)^2(n-2)H^3}|\partial M|,
\end{equation} which proves the asserted inequality.

To conclude, if the equality holds, then one has
\begin{eqnarray}
\nabla^2\Delta f=\frac{\Delta(\Delta f)}{n}g&=&-\frac{R}{n(n-1)}\Delta fg,\nonumber
\end{eqnarray} 
\begin{eqnarray}
\frac{(n+2)R^3}{2n(n-1)}\int_Mf^2\Delta fdV_g&=&0,\nonumber
\end{eqnarray} and
\begin{eqnarray}
\label{kjhy1}
\frac{3(n+2)R}{4}\int_M(\Delta f)^2dV_g&=&0.
\end{eqnarray} Finally, if $R$ is not zero, we conclude from (\ref{kjhy1}) and (\ref{laplaciano}) that $f$ is constant (non-zero), but it must vanish along the boundary, which leads to a contradiction. Hence, $R=0$ and $\Delta f=-\frac{n}{n-1}$ on $M^n.$ This finishes the proof of the theorem. 

\end{proof}

\subsection{Proof of Theorem \ref{th3}}

\begin{proof}
By using the classical Bochner's formula, we have
\begin{eqnarray*}
2fRic(\nabla u,\nabla u)&=&f\Delta|\nabla u|^2-2f|\nabla^2 u|^2-2f\langle\nabla\Delta u,\nabla u\rangle\\
 &=&f\Delta|\nabla u|^2-2f|\nabla^2u|^2+2n\kappa f|\nabla u|^2.
\end{eqnarray*} From this, it follows that
\begin{eqnarray}\label{th3-eq2}
2\int_{M}f\left[Ric-(n-1)\kappa g\right](\nabla u, \nabla u)dV_g&=&\int_Mf\Delta|\nabla u|^2dV_g-2\int_Mf|\nabla^2u|^2dV_g\nonumber\\
 &&+2\kappa \int_M f |\nabla u|^2 dV_g.
\end{eqnarray} Moreover, by Green's identity and the fact that $f=0$ on $\partial M,$ one sees that
\begin{eqnarray*}
\int_Mf\Delta|\nabla u|^2dV_g&=&\int_M|\nabla u|^2\Delta fdV_g\nonumber\\&&+\int_{\partial M}\Big(f\left\langle\nabla
|\nabla u|^2,-\frac{\nabla f}{|\nabla f|}\right\rangle-|\nabla u|^2\left\langle\nabla f,-\frac{\nabla f}{|\nabla f|}\right\rangle\Big)dS\\
 &=&-\frac{R}{n-1}\int_Mf|\nabla u|^2dV_g-\frac{n}{n-1}\int_M|\nabla u|^2dV_g+\frac{1}{H}\int_{\partial M}|\nabla u|^2dS.
\end{eqnarray*} This, together with (\ref{th3-eq2}), yields

\begin{eqnarray}\label{th3-eq3}
2\int_{M}f\left[Ric-(n-1)\kappa g\right](\nabla u, \nabla u)dV_g&=&-\frac{R}{n-1}\int_Mf|\nabla u|^2 dV_g-\frac{n}{n-1}\int_M|\nabla u|^2dV_g\nonumber\\
 &&+\frac{1}{H}\int_{\partial M}|\nabla u|^2dS-2\int_Mf|\nabla^2u|^2dV_g\nonumber\\
 &&+2\kappa\int_Mf|\nabla u|^2dV_g\nonumber\\
 &=&-\frac{R}{n-1}\int_Mf|\nabla u|^2dV_g-\frac{n}{n-1}\int_M|\nabla u|^2dV_g\nonumber\\
 &&+\frac{1}{H}\int_{\partial M}|\nabla u|^2dS-2\int_Mf|\nabla^2u+\kappa ug|^2dV_g\nonumber\\
 &&-2n\kappa^2\int_Mf u^2dV_g+2\kappa\int_Mf |\nabla u|^2dV_g,
\end{eqnarray} which compared with Lemma \ref{lemplk1} gives

\begin{eqnarray}\label{th3-eq4}
\frac{1}{H}\int_{\partial M}\left[|\nabla_{\partial M}\eta|^{2}-(n-1)\kappa\eta^{2}\right]dS &=&-\int _{M} f |\nabla^{2}u+\kappa ug|^{2}dV_g-\frac{1}{n-1}\int_M|\nabla u|^2 dV_g\nonumber \\
 &&-\frac{R-2(n-1)\kappa}{n-1}\int_Mf|\nabla u|^2dV_g\\
 &&+\frac{1}{H}\int_{\partial M}|\nabla u|^2 dS -2n\kappa^2\int_Mfu^2 dV_g\nonumber\\ 
		&&-\kappa\left(R-n(n-1)\kappa\right)\int_{M}fu^{2}dV_g-n\kappa \int_Mu^2 dV_g.\nonumber
		\end{eqnarray}
Now, choosing $\kappa$ and $u$ as in (\ref{eqkl120}), we conclude that $u$ must satisfy
\begin{eqnarray*}
\left\{\begin{array}{rc}
\Delta u +n\kappa u = 0 &\mbox{ in } M,\\
u= -\frac{n}{n-1} & \mbox{ on } \partial M.
\end{array}\right.
\end{eqnarray*} Plugging this fact into (\ref{th3-eq4}) we infer

\begin{eqnarray}\label{th3-eq5}
-\frac{nR}{(n-1)^2H}|\partial M|&=&-\int _{M} f |\nabla^{2}\Delta f+\frac{R\Delta f}{n(n-1)}g|^{2}dV_g-\frac{R^2}{(n-1)^3}\int_M|\nabla f|^2 dV_g\nonumber \\
 &&-\frac{(n-2)R^3}{n(n-1)^3}\int_Mf|\nabla f|^2dV_g+\frac{R^2}{(n-1)^2H^3}|\partial M|\\
 &&-\frac{2R^2}{n(n-1)^2}\int_Mf(\Delta f)^2dV_g-\frac{R}{n-1}\int_M(\Delta f)^2 dV_g.\nonumber
\end{eqnarray} Since $$\int_M|\nabla f|^2dV_g=-\int_Mf\Delta fdV_g\,\,\,\,\mbox{ and }\,\,\,\,\int_Mf|\nabla f|^2dV_g=-\frac{1}{2}\int_Mf^2\Delta fdV_g,$$ we then have
\begin{eqnarray}
-\frac{(nH^2+R)R}{(n-1)^2H^3}|\partial M|&=&-\int _{M} f |\mathring{\nabla^{2}}\Delta f|^{2}dV_g+\frac{R^2}{(n-1)^3}\int_Mf\Delta fdV_g\nonumber \\
 &&+\frac{(n-2)R^3}{2n(n-1)^3}\int_Mf^2\Delta fdV_g-\frac{2R^2}{n(n-1)^2}\int_Mf(\Delta f)^2dV_g\nonumber\\
 &&-\frac{R}{n-1}\int_M(\Delta f)^2dV_g\nonumber\\
 &=&-\int _{M} f |\mathring{\nabla^{2}}\Delta f|^{2}dV_g+\frac{R^2}{(n-1)^3}\int_Mf\Delta fdV_g\nonumber \\
 &&+\frac{(n-2)R^3}{2n(n-1)^3}\int_Mf^2\Delta fdV_g-\frac{2R^2}{n(n-1)^2}\int_Mf\Delta f\left(-\frac{Rf+n}{n-1}\right)dV_g\nonumber\\
 &&-\frac{R}{n-1}\int_M(\Delta f)^2dV_g\nonumber\\
 &=&-\int _{M} f |\mathring{\nabla^{2}}\Delta f|^{2}dV_g+\frac{3R^2}{(n-1)^3}\int_Mf\Delta fdV_g\nonumber \\
 &&+\frac{(n+2)R^3}{2n(n-1)^3}\int_Mf^2\Delta fdV_g-\frac{R}{n-1}\int_M(\Delta f)^2dV_g.\nonumber
\end{eqnarray} Taking into account that

\begin{equation*}
f\Delta f=-\frac{n-1}{R}(\Delta f)^2-\frac{n}{R}\Delta f,
\end{equation*} one sees that 

\begin{eqnarray}
\label{th3-eq61}
-\frac{(nH^2+R)R}{(n-1)^2H^3}|\partial M|&=&-\int _{M} f |\mathring{\nabla^{2}}\Delta f|^{2} dV_g-\frac{(n+2)R}{(n-1)^2}\int_M(\Delta f)^2 dV_g\nonumber \\
 &&-\frac{3nR}{(n-1)^3}\int_M\Delta f dV_g+\frac{(n+2)R^3}{2n(n-1)^3}\int_Mf^2\Delta f dV_g.
\end{eqnarray} Consequently,

\begin{eqnarray}\label{th3-eq63}
-\frac{[n(n+2)H^2+(n-1)R]R}{(n-1)^3H^3}|\partial M|&=&-\int _{M} f |\mathring{\nabla^{2}}\Delta f|^{2}dV_g-\frac{(n+2)R}{(n-1)^2}\int_M(\Delta f)^2dV_g\nonumber \\
 &&+\frac{(n+2)R^3}{2n(n-1)^3}\int_Mf^2\Delta fdV_g.
\end{eqnarray}

Next, since  $R>0$ and $f\geq0,$ it follows from (\ref{laplaciano}) that $-\Delta f\geq 0$ and hence, by Holder's inequality, we achieve
\begin{eqnarray*}
\left(\int_M(f\sqrt{-\Delta f})(\sqrt{-\Delta f})dV_g\right)^2&\leq&\int_Mf^2(-\Delta f)dV_g\cdot\int_M(-\Delta f)dV_g,
\end{eqnarray*} so that
\begin{eqnarray*}
\left(\int_Mf\Delta fdV_g\right)^2&\leq&-\frac{1}{H}|\partial M|\int_Mf^2\Delta fdV_g.
\end{eqnarray*} 
Thereby, one has
\begin{eqnarray*}
-\frac{1}{H}|\partial M|\int_Mf^2\Delta fdV_g&\geq&\left(-\frac{n-1}{R}\int_M(\Delta f)^2dV_g-\frac{n}{R}\int_M\Delta fdV_g\right)^2\\
 &=&\left(-\frac{n-1}{R}\int_M(\Delta f)^2dV_g+\frac{n}{HR}|\partial M|\right)^2\\
 &=&\frac{(n-1)^2}{R^2}\left(\int_M(\Delta f)^2dV_g\right)^2-\frac{2n(n-1)}{HR^2}|\partial M|\int_M(\Delta f)^2dV_g\\
 &&+\frac{n^2}{H^2R^2}|\partial M|^2,
\end{eqnarray*} which can be rewritten as

\begin{eqnarray}\label{th3-eq64}
\int_Mf^2\Delta fdV_g&\leq&-\frac{(n-1)^2H}{R^2|\partial M|}\left(\int_M(\Delta f)^2dV_g\right)^2+\frac{2n(n-1)}{R^2}\int_M(\Delta f)^2dV_g\nonumber\\
 &&-\frac{n^2}{HR^2}|\partial M|.
\end{eqnarray}

Substituting (\ref{th3-eq64}) into (\ref{th3-eq63}), one obtains that
\begin{eqnarray*}
-\frac{[n(n+2)H^2+(n-1)R]R}{(n-1)^3H^3}|\partial M|&\leq&-\int _{M} f |\mathring{\nabla^{2}}\Delta f|^{2}dV_g-\frac{(n+2)R}{(n-1)^2}\int_M(\Delta f)^2dV_g\\
 &&-\frac{(n+2)HR}{2n(n-1)|\partial M|}\left(\int_M(\Delta f)^2dV_g\right)^2\nonumber\\&&+\frac{(n+2)R}{(n-1)^2}\int_M(\Delta f)^2dV_g\\
 &&-\frac{n(n+2)R}{2(n-1)^3H}|\partial M|.
\end{eqnarray*} Rearranging terms we get
\begin{eqnarray}\label{th3-eq65}
-\frac{[n(n+2)H^2+2(n-1)R]R}{2(n-1)^3H^3}|\partial M|&\leq&-\int _{M} f |\mathring{\nabla^{2}}\Delta f|^{2}dV_g\nonumber\\
 &&-\frac{(n+2)HR}{2n(n-1)|\partial M|}\left(\int_M(\Delta f)^2dV_g\right)^2.
\end{eqnarray}

Again, by Holder's inequality, one sees that
\begin{eqnarray*}
Vol(M)\int_M(\Delta f)^2dV_g&\geq&\left(\int_M\Delta fdV_g\right)^2=\frac{1}{H^2}|\partial M|^2.
\end{eqnarray*} This jointly with (\ref{th3-eq65}) gives
\begin{eqnarray}\label{th3-eq67}
-\frac{\left(n(n+2)H^2+2(n-1)R\right)R}{2(n-1)^3H^3}|\partial M|&\leq&-\int _{M} f |\mathring{\nabla^{2}}\Delta f|^{2}dV_g-\frac{(n+2)R}{2n(n-1)H^3}\frac{|\partial M|^3}{(Vol(M))^2}\nonumber\\
 &\leq&-\frac{(n+2)R}{2n(n-1)H^3}\frac{|\partial M|^3}{(Vol(M))^2}.
\end{eqnarray} Thus, since $R>0,$ one concludes that

\begin{equation*}
Vol(M)\geq\frac{n-1}{n}\sqrt{\frac{n(n+2)}{n(n+2)H^2+2(n-1)R}}\, |\partial M|,
\end{equation*} which proves (\ref{th3-est01}).

Finally, if the equality holds, then $$\nabla^2\Delta f=\frac{\Delta(\Delta f)}{n}g=-\frac{R}{n(n-1)}\Delta f g.$$ Moreover, since $\Delta f=-\frac{n}{n-1}$ on the boundary $\partial M$ (i.e., $\Delta f$ is constant on $\partial M$), we may use Theorem B  (II) of \cite{Reilly} to deduce that $(M^n,\,g)$ has constant sectional curvature $\frac{R}{n(n-1)}$ and this forces $(M^n,g)$ to be an Einstein manifold. Finally, it suffices to invoke Theorem 1.1 of \cite{Miao-Tam 2011} to conclude that $(M^n,\,g)$ is isometric to a geodesic ball in a simply connected space form $\Bbb{S}^n.$ Hence, Theorem \ref{th3} is proved. 
\end{proof}

To conclude, we will present the proof of the estimate (\ref{eqV3}) stated in Remark \ref{remA}. More precisely, we have the following proposition.

\begin{proposition}
\label{Propth3}
Let $\big(M^{n},\,g,\,f)$ be a connected compact oriented critical metric with connected boundary $\partial M$ and negative scalar curvature. Then we have:

 $$Vol(M)\geq \frac{(n-1)H}{nH^2 - R}\,|\partial M|.$$
\end{proposition}

\begin{proof}
Initially, taking $\kappa=\frac{R}{n(n-1)}<0$ in Theorem \ref{th2}, one obtains that
\begin{eqnarray}\label{th3-eq8}
H\int_{\partial M}\SP{\nabla f,\nabla u}^2 dS-\frac{R}{nH}\int_{\partial M}\eta^2 dS&\geq&\frac{1}{n-1}\int_M|\nabla u|^2dV_g-\frac{R}{n-1}\int_Mu^2dV_g\nonumber\\
 &&+\frac{2R^2}{(n-1)^2}\int_Mfu^2dV_g-\frac{(3n-2)R}{n(n-1)}\int_Mf|\nabla u|^2dV_g\nonumber\\
 &\geq&-\frac{R}{n-1}\int_Mu^2dV_g.
\end{eqnarray} Next, choosing $u=\Delta f=-\dfrac{Rf+n}{n-1},$ we infer that $u$ must satisfy
\begin{eqnarray*}
\left\{\begin{array}{rc}
\Delta u +n\kappa u = 0 &\mbox{ in } M,\\
u= -\frac{n}{n-1} & \mbox{ on } \partial M.
\end{array}\right.
\end{eqnarray*} Plugging this into (\ref{th3-eq8}), one sees that
\begin{equation*}
H\frac{R^2}{(n-1)^2}\int_{\partial M}|\nabla f|^4 dS-\frac{nR}{(n-1)^2H}|\partial M|\geq-\frac{R}{n-1}\int_M(\Delta f)^2dV_g.
\end{equation*} Since $H=\frac{1}{|\nabla f|},$ we deduce
\begin{equation}
\label{th3-eq9}
\frac{R(R-nH^2)}{(n-1)^2H^3}|\partial M|\geq-\frac{R}{n-1}\int_M(\Delta f)^2dV_g.
\end{equation}

On the other hand, by Holder's inequality we get
\begin{eqnarray*}
Vol(M)\int_M(\Delta f)^2dV_g&\geq&\left(\int_M\Delta fdV_g\right)^2=\left(\int_{\partial M}\Big\langle\nabla f,-\frac{\nabla f}{|\nabla f|}\Big\rangle dS \right)^2\\
 &=&\frac{1}{H^2}|\partial M|^2.
\end{eqnarray*} This jointly with (\ref{th3-eq9}) yields
\begin{eqnarray*}
Vol(M)\geq\frac{(n-1)H}{nH^2-R}|\partial M|,
\end{eqnarray*} which proves the asserted result.
\end{proof}


\begin{thebibliography}{BB}

\bibitem{BBR} Baltazar, H., Batista, R. and Ribeiro Jr., E.: Geometric inequalities for critical metrics of the volume functional. {\it Annali di Matematica Pura ed Appl.}, 201 (2022) 1463-1480.

\bibitem{BDR}  Baltazar, H., Di\'ogenes, R. and Ribeiro Jr., E.: Volume functional of compact $4$-manifolds with a prescribed boundary metric. {\it J. Geom. Analysis.} 31 (2021) 4703-4720.

\bibitem{Baltazar-Diogenes-Ribeiro} Baltazar, H., Di\'ogenes, R. and Ribeiro Jr., E.: Isoperimetric inequality and Weitzenboöck type formula for critical metrics of the volume. {\it Israel J. Math.} 234 (2019) 309-329.

\bibitem{Baltazar-Ribeiro Jr. Ricci paralelo} Baltazar, H. and Ribeiro Jr., E.: Critical metrics of the volume functional on manifolds with boundary. {\it Proc. Amer. Math. Soc}. 145 (2017) 3513-3523.

\bibitem{Bach-Flat} Barros, A., Di\'ogenes, R. and Ribeiro Jr., E.: Bach-flat critical metrics of the volume functional on 4-dimensional manifolds with boundary. {\it J. Geom. Analysis.} 25 (2015) 2698-2715.

\bibitem{BS} Barros, A. and Silva, A.: Rigidity for critical metrics of the volume functional. {\it Math. Nachr.} 291 (2019) 709-719.

\bibitem{Batista-Diogenes-Ranieri-Ribeiro JR} Batista, R., Di\'ogenes, R., Ranieri, M. and Ribeiro Jr., E.: Critical metrics of the volume functional on compact three-manifolds with smooth boundary. {\it J. Geom. Analysis.} 27 (2017) 1530-1547.

\bibitem{Besse} Besse, A. L.: Einstein Manifolds. Springer-Verlag, Berlin (1987).

\bibitem{Brendle} Brendle, S. and Chen, S.-Y.: An existence theorem for the Yamabe problem on manifolds with boundary. {\it J. Eur. Math. Soc.} 16 (2014) 991-1016.

\bibitem{Catino 2016} Catino, G.: Integral pinched shrinking Ricci solitons. {\it Adv. Math.} 303 (2016) 279-294.
		
\bibitem{Chen-Wang-Yau} Chen, P.-N., Wang, M.-T. and Yau, S.-T.:  Minimizing properties of critical points of quasi-local energy. {\it Comm. Math. Phys.}  (3) 329 (2014) 919–935.

\bibitem{Corvino-Eichmair-Miao} Corvino, J.: Eichmair, M. and Miao, P.: Deformation of scalar curvature and volume. {\it Math. Ann.} 357 (2013) 551-584.
		
\bibitem{EscobarJDG} Escobar, J.: The Yamabe problem on manifolds with boundary. {\it J. Diff. Geom.} 34 (1992) 21-84.		

\bibitem{FY} Fang, Y. and Yuan, W.: Brown–York mass and positive scalar curvature II: Besse's conjecture and related problems. {\it Ann. Glob. Anal. Geom.} 56 (2019) 1-15.

\bibitem{FST} Fan, X.-Q., Shi, Y.-G. and  Tam, L.-F.: Large-sphere and small-sphere limits of the Brown-York mass. {\it Comm. Anal. Geom.} 17 (2009) 37-72.

\bibitem{KS} Kim, J. and Shin, J.: Four-dimensional static and related critical spaces with harmonic curvature. {\it Pacific J. Math}. 295 (2018) 429-462.
		
\bibitem{Kwong-Miao} Kwong, K. and Miao, P.:  A  functional  inequality on the boundary of static manifolds. {\it Asian J. Math.} (4) 21 (2017) 687-696.

\bibitem{Li-Xia} Li, J. and Xia, C.:  An integral formula and its applications on sub-static manifolds. {\it J. Differential Geom.} 113 (2019) 493-518.

\bibitem{Miao-Tam 2011} Miao, P. and Tam, L.-F.: Einstein and conformally flat critical metrics of the volume functional. {\it Trans. Amer. Math. Soc.} 363 (2011) 2907-2937.

\bibitem{Miao-Tam 2009} Miao, P. and Tam, L.-F.: On the volume functional of compact manifolds with boundary with constant scalar curvature. {\it Calc. Var. PDE.} 36 (2009) 141-171.

\bibitem{Miao-Tam-N.-Q} Miao, P., Tam, L.-F. and Xie, N.-Q.: Critical points of Wang-Yau quasi-local energy. {\it Ann. Henri Poincar\'e.} (5) 12 (2011) 987–1017.

\bibitem{MW} Miao, P. and Wang, X.: Boundary effect of Ricci curvature. {\it J. Diff. Geom.} 103 (2016) 59-82.

\bibitem{Qiu-Xia} Qiu, G. and Xia, C.: A generalization of Reilly’s formula and its applications to a new Heintze-Karcher type inequality. {\it Int. Math. Res. Not. IMRN}. 17 (2015) 7608-7619.
	
\bibitem{Reilly1977} Reilly, R.: Applications of the Hessian operator in a Riemannian manifold. {\it Indiana Univ. Math. J.} 26 (1977) 459-472.

\bibitem{Reilly} Reilly, R. C.: Geometric applications of the solvability of Neumann problems on a Riemannian manifold. {\it Arch. Rational Mech. Anal.} (1) 75 (1980)  23-29.

\bibitem{Ros} Ros, A.: Compact hypersurfaces with constant higher order mean curvatures. {\it Revista Matem\'atica Iberoamericana}. 3 (1987) 447-453.

\bibitem{SW} Sheng, W. and Wang, L.: Critical metrics with cyclic parallel Ricci tensor for volume functional on manifolds with boundary. {\it Geometriae Dedicata}. 201 (2019) 243-251.

\bibitem{Schoen2} Schoen, R.: Variational theory for the total scalar curvature functional for Riemannian metrics and related topics. \textit{Topics in calculus of variations. Lecture Notes in Math.} 1365 (1989) 120-154. 

\bibitem{X. Wang} Wang, X.: On compact Riemannian manifolds with convex boundary and Ricci curvature bounded from below. {\it J. Geom. Analysis.} 31 (2021) 3988-4003.

\bibitem{Wang-Yau}Wang, M.-T. and Yau, S.-T.:  Quasilocal mass in general relativity. {\it Phys. Rev.Lett.} (2) 102 (2009) 021101.

\bibitem{Wang-Yau isometric}Wang, M.-T. and Yau, S.-T.: Isometric embeddings into the Minkowski space and new quasi-local mass. {\it Comm. Math. Phys.} (3) 288  (2009) 919–942.

\bibitem{Xia0} Xia, C: A Minkowski type inequality in space forms. {\it Calc. Var. PDE.} (2016) 55:96.

\bibitem{yuan} Yuan, W.: Volume comparison with respect to scalar curvature. ArXiv:1609.08849 [math.DG] (2016)
		
	\end{thebibliography}
\end{document}